\documentclass[preprint,12pt]{elsarticle}
\usepackage{amssymb}
\usepackage{amsmath,graphicx}
\usepackage{amsfonts}

\newtheorem{theorem}{Theorem}

\newtheorem{corollary}[theorem]{Corollary}

\newtheorem{definition}[theorem]{Definition}

\newtheorem{proposition}[theorem]{Proposition}
\newtheorem{remark}[theorem]{Remark}

\newenvironment{proof}[1][Proof]{\noindent\textbf{#1.} }{\ \rule{0.5em}{0.5em}}
\begin{document}

\begin{frontmatter}



\title{A characterization of involutes and evolutes of a given curve in $\mathbb{E}^{n}$}


\author{G\"{u}nay \"{O}ZT\"{U}RK}
\address{Department of Mathematics, Kocaeli University, Kocaeli, Turkey \\
} \ead{ogunay@kocaeli.edu.tr}
\author{Kadri ARSLAN, Bet\"{u}l BULCA}
\address{Department of Mathematics, Uluda\u{g} University, Bursa, Turkey\\
} \ead{arslan@uludag.edu.tr, bbulca@uludag.edu.tr}

\begin{abstract}
The orthogonal trajectories of the first tangents of the curve are
called the involutes of $x$. The hyperspheres which have higher
order contact with a curve $x$ are known osculating hyperspheres of
$x$. The centers of osculating hyperspheres form a curve which is
called generalized evolute of the given curve $x$ in $n$-dimensional
Euclidean space $\mathbb{E}^{n}$. In the present study, we give a
characterization of involute curves of order $k$ (resp. evolute
curves) of the given curve $x$ in $n$-dimensional Euclidean space
$\mathbb{E}^{n}$. Further, we obtain some results on these type of
curves in $\mathbb{E}^{3}$ and $\mathbb{E}^{4}$, respectively.
\end{abstract}


\end{frontmatter}



\section{Introduction}

Let $x=x(t):I\subset \mathbb{R}\rightarrow \mathbb{E}^{n}$ be a regular
curve in $\mathbb{E}^{n}$, $($i.e., $\left \Vert x^{\prime }(t)\right \Vert
\neq 0)$. Then $x$ is called a \textit{Frenet curve of \ osculating order }$%
d $, $(2\leq d\leq n)$ if $x^{ \prime }(t),$ $x^{ \prime \prime }(t),$...,$%
x^{(d)}(t)$ are linearly independent and $x^{ \prime }(t),$ $x^{ \prime
\prime }(t),$...,$x^{(d+1)}(t)$ linearly dependent for all $t$ in $I$ \cite%
{Va}. In this case, $Im(x)$ lies in an $d$-dimensional Euclidean subspace of
$\mathbb{E}^{n+1}$. To each Frenet curve of rank $d$ there can be associated
orthonormal $d$-frame $V_{1}=\frac{x^{ \prime }(t)}{\left
\Vert x^{ \prime
}(t)\right \Vert },V_{2},V_{3}...,V_{d}$ along $x$, the Frenet $d$-frame,
and $d-1$ functions $\kappa _{1},\kappa _{2},...,\kappa _{d-1}$:$%
I\longrightarrow \mathbb{R}$, the Frenet curvature, such that%
\begin{equation}
\left[
\begin{array}{c}
V_{1}^{^{\prime }} \\
V_{2}^{^{\prime }} \\
V_{3}^{^{\prime }} \\
... \\
V_{d}^{^{\prime }}%
\end{array}%
\right] =v\left[
\begin{array}{ccccc}
0 & \kappa _{1} & 0 & ... & 0 \\
-\kappa _{1} & 0 & \kappa _{2} & ... & 0 \\
0 & -\kappa _{2} & 0 & ... & 0 \\
... &  &  &  & \kappa _{d-1} \\
0 & 0 & ... & -\kappa _{d-1} & 0%
\end{array}%
\right] \left[
\begin{array}{c}
V_{1} \\
V_{2} \\
V_{3} \\
... \\
V_{d}%
\end{array}%
\right]  \label{a1.1}
\end{equation}%
where, $v=\left \Vert x^{ \prime }(t)\right \Vert $ is the speed of the
curve $x$. In fact, to obtain $V_{1},V_{2},V_{3}...,V_{d},$ $(2\leq d\leq n)$
it is sufficient to apply the Gram-Schmidt orthonormalization process to $%
x^{\prime }(t),$ $x^{\prime \prime }(t),$...,$x^{(d)}(t)$. Moreover, the
functions $\kappa _{1},\kappa _{2},...,\kappa _{d-1}$ are easily obtained as
by-product during this calculation.

More precisely, $V_{1},V_{2},V_{3}...,V_{d}$ and $\kappa _{1},\kappa
_{2},...,\kappa _{d-1}$ are determined by the following formulas:%
\begin{eqnarray}
E_{1}(t) &:&=x^{\prime }(t)\ \ \text{\ \ };V_{1}:=\frac{E_{1}(t)}{\left
\Vert E_{1}(t)\right \Vert },  \notag \\
E_{\alpha }(t) &:&=x^{(\alpha )}(t)-\sum_{i=1}^{\alpha -1}<x^{(\alpha
)}(t),E_{i}(t)>\frac{E_{i}(t)}{\left \Vert E_{i}(t)\right \Vert ^{2}},\text{
}  \label{a1.2} \\
V_{\alpha } &:&=\frac{E_{\alpha }(t)}{\left \Vert E_{\alpha }(t)\right \Vert
},2\leq \alpha \leq n  \notag
\end{eqnarray}%
and%
\begin{equation}
\kappa _{\delta }(s):=\frac{\left \Vert E_{\delta +1}(t)\right \Vert }{\left
\Vert E_{\delta }(t)\right \Vert \left \Vert E_{1}(t)\right \Vert },
\label{a1.3}
\end{equation}%
respectively, where $\delta \in \left \{ 1,2,3,...,d-1\right \} $ (see, \cite%
{Gl}). For the case $d=n$, the Frenet curve $x$ is called a \textit{generic
curve} \cite{Va}.

The osculating hyperplanes of a generic curve $x$ at $t$ is the subspace
generated by $\left \{ V_{1},V_{2},V_{3}...,V_{n}\right \} $ that passes
through $x(t)$. The unit vector $V_{n}(t)$ is called \textit{binormal vector}
of $x$ at $t$. The \textit{normal hyperplane} of $x$ at $t$ is defined to be
the one generated by $\left \{ V_{2},V_{3}...,V_{n}\right \} $ passing
through $x(t)$ \cite{FC}.

A Frenet curve of rank $d$ for which the first Frenet curvature $\kappa _{1}$
is constant is called a Salkowski curve \cite{S} (or T.C-curve \cite{KAO}).
Further, a Frenet curve of rank $d$ for which $\kappa _{1},\kappa
_{2},...,\kappa _{d-1}$ are constant is called (\textit{circular}) \textit{%
helix} or $W$\textit{-curve} \cite{KL}. Meanwhile, a Frenet curve of rank $d$
with constant curvature ratios $\frac{\kappa _{2}}{\kappa _{1}},\frac{\kappa
_{3}}{\kappa _{2}},\frac{\kappa _{4}}{\kappa _{3}},...,\frac{\kappa _{d-1}}{%
\kappa _{d-2}}$ is called a \textit{ccr-curve} (see, \cite{OAH}, \cite{Mo}).
A ccr-curve in $\mathbb{E}^{3}$ is known as generalized helix.

Given a generic curve $x$ in $\mathbb{E}^{4}$, the Frenet $4$-frame, $%
V_{1},V_{2},V_{3},V_{4}$ and the Frenet curvatures $\kappa _{1},\kappa
_{2},\kappa _{3}$ are given by%
\begin{eqnarray}
V_{1}(t) &=&\frac{x^{\prime }(t)}{\left \Vert x^{\prime }(t)\right \Vert }
\notag \\
V_{4}(t) &=&\frac{x^{\prime }(t)\wedge x^{\prime \prime }(t)\wedge x^{\prime
\prime \prime }(t)}{\left \Vert x^{\prime }(t)\wedge x^{\prime \prime
}(t)\wedge x^{\prime \prime \prime }(t)\right \Vert }  \label{a1.4} \\
V_{3}(t) &=&\frac{V_{4}(t)\wedge x^{\prime }(t)\wedge x^{\prime \prime }(t)}{%
\left \Vert V_{4}(t)\wedge x^{\prime }(t)\wedge x^{\prime \prime }(t)\right
\Vert }  \notag \\
V_{2}(t) &=&\frac{V_{3}(t)\wedge V_{4}(t)\wedge x^{\prime }(t)}{\left \Vert
V_{3}(t)\wedge V_{4}(t)\wedge x^{\prime }(t)\right \Vert }  \notag
\end{eqnarray}%
and%
\begin{equation}
\kappa _{1}(t)=\frac{\left \langle V_{2}(t),x^{\prime \prime }(t)\right
\rangle }{\left \Vert x^{\prime }(s)\right \Vert ^{2}},\text{ }\kappa
_{2}(t)=\frac{\left \langle V_{3}(t),x^{\prime \prime \prime }(t)\right
\rangle }{\left \Vert x^{\prime }(t)\right \Vert ^{3}\kappa _{1}(t)},\kappa
_{3}(t)=\frac{\left \langle V_{4}(t),x^{\prime \prime \prime \prime
}(t)\right \rangle }{\left \Vert x^{\prime }(t)\right \Vert ^{4}\kappa
_{1}(t)\kappa _{2}(t)}.  \label{a1.5}
\end{equation}%
respectively, where $\wedge $ is the exterior product in $\mathbb{E}^{4}$
\cite{Gl}.

This paper is organized as follows: Section 2 gives some basic concepts of
the involute curves of order $k$ in $\mathbb{E}^{n}$. Section 3 explains
some geometric properties about the involute curves of order $k$ in $\mathbb{%
E}^{3}$, where $k=1,2$. Section 4 tells about the involute curves of order $%
k $ in $\mathbb{E}^{4}$, where $k=1,2,3$. Further these sections provides
some properties and results of these type of curves. In the final section we
consider generalized evolute curves in $\mathbb{E}^{n}$. Moreover, we
present some results of generalized evolute curves in $\mathbb{E}^{3}$ and $%
\mathbb{E}^{4}$, respectively.

\section{Involute curves of order $k$ in $\mathbb{E}^{n}$}

\begin{definition}
Let $x=x(s)$ be a regular generic curve in $\mathbb{E}^{n}$ given with the
arclength parameter $s$ $\left( i.e.,\left \Vert x^{\prime }(s)\right \Vert
=1\right)$. Then the curves which are orthogonal to the system of $k$%
-dimensional osculating hyperplanes of $x$, are called the \textit{involutes
of order }$k$ \cite{DZ} (or, $k^{th}$ involute \cite{He}) of the curve $x$.
For simplicity, we call the \textit{involutes of order }$1$, simply the
involutes of the given curve.
\end{definition}

In order to find the parametrization of involutes $\overline{x}$ of order $k$
of the curve $x$, we put
\begin{equation}
\overline{x}(s)=x(s)+\sum \limits_{\alpha =1}^{k}\lambda _{\alpha
}(s)V_{\alpha }(s),\text{ }k\leq n-1  \label{b1.1}
\end{equation}%
where $\lambda _{\alpha }$ is a differentiable function and $s$ is the
parameter of $\overline{x}$ which is not necessarily an arclength parameter.
The differentiation of the equation (\ref{b1.1}) and the Frenet formulae (%
\ref{a1.1}) give the following equation
\begin{eqnarray}
\overline{x}^{\prime }(s) &=&\left( 1+\lambda _{1}^{\prime }-\kappa
_{1}\lambda _{2}\right) (s)V_{1}(s)  \notag \\
&&+\sum \limits_{\alpha =2}^{k-1}\left( \lambda _{\alpha }^{\prime }-\lambda
_{\alpha +1}\kappa _{\alpha }+\lambda _{\alpha -1}\kappa _{\alpha -1}\right)
(s)V_{\alpha }(s)  \label{b1.2} \\
&&+\left( \lambda _{k}^{\prime }+\lambda _{k-1}\kappa _{k-1}\right)
(s)v_{\alpha }(s)+\kappa _{k}(s)\lambda _{k}(s)V_{k+1}(s).  \notag
\end{eqnarray}%
Furthermore, the involutes $\overline{x}$ of order $k$ of the curve $x$ are
determined by
\begin{equation*}
\left \langle \overline{x}^{\prime }(s),V_{j}(s)\right \rangle =0,1\leq
j\leq k\leq n-1.
\end{equation*}%
This condition is satisfied if and only if
\begin{eqnarray}
1+\lambda _{1}^{\prime }-\kappa _{1}\lambda _{2} &=&0,  \notag \\
\lambda _{\alpha }^{\prime }-\lambda _{\alpha +1}\kappa _{\alpha }+\lambda
_{\alpha -1}\kappa _{\alpha -1} &=&0,  \label{b1.3} \\
\lambda _{k}^{\prime }+\lambda _{k-1}\kappa _{k-1} &=&0,  \notag
\end{eqnarray}
where $2\leq \alpha \leq n-1$. Consequently, the involutes of order $k$ of a
regular generic curve $x$ are represented by the formulas (\ref{b1.3}), and
when $\lambda _{\alpha }$ are chosen in this way, $\lambda _{k}$ does not
vanish identically and $\overline{V}_{1}(s)=\pm V_{k+1}$ whenever $\lambda
_{k}\neq 0$ \cite{He}.

\section{Involutes in $\mathbb{E}^{3}$}

In the present section we consider involutes of order 1 and of order 2 of
curves in Euclidean 3-space $\mathbb{E}^{3}$, respectively.

\subsection{Involutes of order $1$ in $\mathbb{E}^{3}$}

\begin{proposition}
Let $x=x(s)$ be a regular curve in $\mathbb{E}^{3}$ given with nonzero
Frenet curvatures $\kappa _{1}$ and $\kappa _{2}.$ Then Frenet curvatures $%
\overline{\kappa }_{1}$ and $\overline{\kappa }_{2}$ of the involute $%
\overline{x}$ of the curve $x$ are given by
\begin{equation}
\overline{\kappa }_{1}=\frac{\sqrt{\kappa _{1}^{2}+\kappa _{2}^{2}}}{\left
\vert \kappa _{1}\right \vert \left \vert s-c\right \vert },\text{ }%
\overline{\kappa }_{2}=\frac{\left( \frac{\kappa _{2}}{\kappa _{1}}\right)
^{\prime }\kappa _{1}^{2}}{\left( \kappa _{1}^{2}+\kappa _{2}^{2}\right)
\left( c-s\right) }.  \label{b1.4}
\end{equation}
\end{proposition}

\begin{proof}
Let $\overline{x}=\overline{x}(s)$ be the involute of the curve $x$ in $%
\mathbb{E}^{3}$. Then by the use of (\ref{b1.2}) with (\ref{b1.3}) we get $%
1+\lambda _{1}^{\prime }(s)=0$, and furthermore $\lambda (s)=\left(
c-s\right) $ for some integral constant $c$. So, we get the following
parametrization
\begin{equation}
\overline{x}(s)=x(s)+(c-s)V_{1}(s).  \label{b1.5}
\end{equation}%
\qquad Further, the differentiation of (\ref{b1.5}) implies that
\begin{eqnarray*}
\overline{x}^{\prime }(s) &=&\varphi V_{2},\text{ }\varphi (s):=\lambda
(s)\kappa _{1}(s) \\
\overline{x}^{\prime \prime }(s) &=&-\varphi \kappa _{1}V_{1}+\varphi
^{\prime }V_{2}+\varphi \kappa _{2}V_{3}, \\
\overline{x}^{\prime \prime \prime }(s) &=&-\left \{ \left( \kappa
_{1}\varphi \right) ^{\prime }+\kappa _{1}\varphi ^{\prime }\right \}
V_{1}+\left \{ \varphi ^{\prime \prime }-\kappa _{1}^{2}\varphi -\kappa
_{2}^{2}\varphi \right \} V_{2}+\left \{ \left( \kappa _{2}\varphi \right)
^{\prime }+\kappa _{2}\varphi ^{\prime }\right \} V_{3}.
\end{eqnarray*}%
\qquad Now, an easy calculation gives
\begin{eqnarray}
\left \Vert \overline{x}^{\prime }(s)\right \Vert &=&\left \vert \varphi
\right \vert =\left \vert (c-s)\kappa _{1}\right \vert ,  \notag \\
\left \Vert \overline{x}^{\prime }(s)\times \overline{x}^{\prime \prime
}(s)\right \Vert &=&\varphi ^{2}\sqrt{\kappa _{1}^{2}+\kappa _{2}^{2}},
\label{b1.6} \\
\left \langle \overline{x}^{\prime }(s)\times \overline{x}^{\prime \prime
}(s),\overline{x}^{\prime \prime \prime }(s)\right \rangle &=&\varphi
^{3}\left( \kappa _{1}\kappa _{2}^{\prime }-\kappa _{2}\kappa _{1}^{\prime
}\right) .  \notag
\end{eqnarray}%
\qquad The parameter $s$ is not the arc length parameter of $\overline{x}$,
so, as is shown in \cite{DZ}, we have%
\begin{equation}
\overline{\kappa }_{1}=\frac{\left \Vert \overline{x}^{\prime }(s)\times
\overline{x}^{\prime \prime }(s)\right \Vert }{\left \Vert \overline{x}%
^{\prime }(s)\right \Vert ^{3}},\text{ }\overline{\kappa }_{2}=\frac{\left
\langle \overline{x}^{\prime }(s)\times \overline{x}^{\prime \prime }(s),%
\overline{x}^{\prime \prime \prime }(s)\right \rangle }{\left \Vert
\overline{x}^{\prime }(s)\times \overline{x}^{\prime \prime }(s)\right \Vert
^{2}}  \label{b1.7}
\end{equation}%
\qquad Hence, from the relations (\ref{b1.6}) and (\ref{b1.7}) we deduce (%
\ref{b1.4}).
\end{proof}

By the use of (\ref{b1.4}) one can get the following result.

\begin{corollary}
$If$ $x=x(s)$ is a cylindrical helix in $\mathbb{E}^{3}$, then the involute $%
\overline{x}$ of $x$ is a planar curve.
\end{corollary}

\subsection{Involutes of order $2$ in $\mathbb{E}^{3}$}

\begin{flushleft}
An involute of order $2$ of a regular curve $x$ in $\mathbb{E}^{3}$ has the
parametrization
\end{flushleft}

\begin{equation}
\overline{x}(s)=x(s)+\lambda _{1}(s)V_{1}(s)+\lambda _{2}(s)V_{2}(s)
\label{b1.8}
\end{equation}%
where $V_{1},V_{2}$ are tangent and normal vectors of $x$ in $\mathbb{E}^{3}$
and $\lambda _{1}$, $\lambda _{2}$ are differentiable functions satisfying
\begin{equation}
\begin{array}{c}
\text{ \ \ \ }\lambda _{1}^{\prime }(s)=\kappa _{1}(s)\lambda _{2}(s)-1, \\
\lambda _{2}^{\prime }(s)=-\lambda _{1}(s)\kappa _{1}(s).%
\end{array}
\label{b1.9}
\end{equation}

We obtain the following result.

\begin{proposition}
Let $x=x(s)$ be a regular curve in $\mathbb{E}^{3}$ given with nonzero
Frenet curvatures $\kappa _{1}$ and $\kappa _{2}$. Then Frenet curvatures $%
\overline{\kappa }_{1}$ and $\overline{\kappa }_{2}$ of the involute $%
\overline{x}$ of order $2$ of the curve $x$ are given by
\begin{equation}
\overline{\kappa }_{1}=\frac{sgn(\kappa _{2})}{\left \vert \lambda
_{2}\right \vert },\text{ }\overline{\kappa }_{2}=\frac{\frac{\kappa _{2}}{%
\kappa _{1}}}{\lambda _{2}}.  \label{b1.10}
\end{equation}
\end{proposition}

\begin{proof}
Let $\overline{x}=\overline{x}(s)$ be the involute of order $2$ of the curve
$x$ in $\mathbb{E}^{3}$. Then by the use of (\ref{b1.2}) with (\ref{b1.3})
we get%
\begin{equation}
\overline{x}^{\prime }(s)=\lambda _{2}(s)\kappa _{2}(s)V_{3}(s),
\label{b1.11}
\end{equation}%
Further, the differentiation of (\ref{b1.11}) implies that
\begin{eqnarray*}
\overline{x}^{\prime }(s) &=&\psi (s)V_{3}(s),\text{ }\psi (s):=\lambda
_{2}(s)\kappa _{2}(s) \\
\overline{x}^{\prime \prime }(s) &=&-\psi (s)\kappa _{2}(s)V_{2}(s)+\psi
^{\prime }(s)V_{3}(s), \\
\overline{x}^{\prime \prime \prime }(s) &=&-\psi (s)\kappa _{1}(s)\kappa
_{2}(s)V_{1}(s)-\left \{ \left( \psi (s)\kappa _{2}(s)\right) ^{\prime
}+\kappa _{2}(s)\psi ^{\prime }(s)\right \} V_{2}(s) \\
&&+\left \{ \psi ^{\prime \prime }(s)+\psi (s)\kappa _{2}^{2}(s)\right \}
V_{3}(s).
\end{eqnarray*}%
Now, an easy calculation gives
\begin{eqnarray}
\left \Vert \overline{x}^{\prime }(s)\right \Vert &=&\left \vert \psi
(s)\right \vert =\left \vert \lambda _{2}(s)\kappa _{2}(s)\right \vert ,
\notag \\
\left \Vert \overline{x}^{\prime }(s)\times \overline{x}^{\prime \prime
}(s)\right \Vert &=&\psi (s)^{2}\kappa _{2}(s),  \label{b1.12} \\
\left \langle \overline{x}^{\prime }(s)\times \overline{x}^{\prime \prime
}(s),\overline{x}^{\prime \prime \prime }(s)\right \rangle &=&\psi
(s)^{3}\kappa _{1}(s)\kappa _{2}^{2}(s).  \notag
\end{eqnarray}%
Hence, from the relations (\ref{b1.7}) and (\ref{b1.12}) we deduce (\ref%
{b1.10}).
\end{proof}

\begin{corollary}
The involute $\overline{x}$ of order $2$ of a generalized helix in $\mathbb{E%
}^{3}$ is also a generalized helix in $\mathbb{E}^{3}$.
\end{corollary}

Solving the system of differential equations (\ref{b1.9}) we get the
following result.

\begin{corollary}
Let $x=x(s)$ be a unit speed Salkowski curve in $\mathbb{E}^{3}$. Then the
involute $\overline{x}$ of order $2$ of the curve $x$ has the
parametrization (\ref{b1.8}) given with the coefficient functions
\begin{eqnarray}
\lambda _{1}(s) &=&c_{1}\sin (\kappa _{1}s)+c_{2}\cos (\kappa _{1}s),  \notag
\\
\lambda _{2}(s) &=&c_{1}\cos (\kappa _{1}s)-c_{2}\sin (\kappa _{1}s)-\frac{1%
}{\kappa _{1}}.  \label{b1.13}
\end{eqnarray}%
where $c_{1}$ and $c_{2}$ are real constants.
\end{corollary}

\section{Involutes in $\mathbb{E}^{4}$}

In the present section we consider involutes of order $k$, $1\leq k\leq 3$
of a given curve $x$ in Euclidean 4-space $\mathbb{E}^{4}$.

\subsection{Involutes of order $1$ in $\mathbb{E}^{4}$}

\begin{proposition}
Let $x=x(s)$ be a regular curve in $\mathbb{E}^{4}$ given with the Frenet
curvatures $\kappa _{1}$, $\kappa _{2}$ and $\kappa _{3}$. Then Frenet $4$%
-frame, $\overline{V}_{1},\overline{V}_{2},\overline{V}_{3}$ and $\overline{V%
}_{4}$ and Frenet curvatures $\overline{\kappa }_{1}$, $\overline{\kappa }%
_{2}$ and $\overline{\kappa }_{3}$ of the involute $\overline{x}$ of the
curve $x$ are given by%
\begin{eqnarray}
\overline{V}_{1}(s) &=&V_{2},  \notag \\
\overline{V}_{2}(s) &=&\frac{-\kappa _{1}V_{1}+\kappa _{2}V_{3}}{\sqrt{%
\kappa _{1}^{2}+\kappa _{2}^{2}}},  \notag \\
\overline{V}_{3}(s) &=&\frac{-\left( \kappa _{2}A-\kappa _{1}C\right) \kappa
_{2}V_{1}-\left( \kappa _{2}A-\kappa _{1}C\right) \kappa _{1}V_{3}+D\left(
\kappa _{1}^{2}+\kappa _{2}^{2}\right) V_{4}}{W\sqrt{\kappa _{1}^{2}+\kappa
_{2}^{2}}},  \label{b1.14} \\
\overline{V}_{4}(s) &=&\frac{D\kappa _{2}V_{1}+D\kappa _{1}V_{3}-\left(
\kappa _{2}A-\kappa _{1}C\right) V_{4}}{W},  \notag
\end{eqnarray}%
and
\begin{eqnarray}
\overline{\kappa }_{1} &=&\frac{\sqrt{\kappa _{1}^{2}+\kappa _{2}^{2}}}{%
\left \vert \varphi \right \vert };\text{ }\varphi :=(c-s)\kappa _{1},
\notag \\
\overline{\kappa }_{2} &=&\frac{W}{\varphi ^{2}\left( \kappa _{1}^{2}+\kappa
_{2}^{2}\right) },  \label{b1.15} \\
\overline{\kappa }_{3} &=&-\frac{\left( \kappa _{2}A-\kappa _{1}C\right)
\left( \kappa _{3}C+D^{\prime }\right) +D\left( \kappa _{2}A^{\prime
}-\kappa _{1}C^{\prime }\right) +D^{2}\kappa _{1}\kappa _{3}}{W\varphi ^{4}%
\overline{\kappa }_{1}\overline{\kappa }_{2}},  \notag
\end{eqnarray}%
respectively, where%
\begin{eqnarray*}
A &=&\kappa _{1}^{\prime }\varphi +2\kappa _{1}\varphi ^{\prime } \\
C &=&\kappa _{2}^{\prime }\varphi +2\kappa _{2}\varphi ^{\prime } \\
D &=&\kappa _{2}\kappa _{3}\varphi
\end{eqnarray*}%
and
\begin{eqnarray}
W&=&\sqrt{D^{2}\left( \kappa _{1}^{2}+\kappa _{2}^{2}\right) +\left( \kappa
_{1}C-\kappa _{2}A\right) ^{2}} \\
&=&\left \vert \varphi \right \vert \sqrt{\kappa _{2}^{2}\kappa
_{3}^{2}\left( \kappa _{1}^{2}+\kappa _{2}^{2}\right) +\left( \kappa
_{1}\kappa _{2}^{\prime }-\kappa _{2}\kappa _{1}^{\prime }\right) ^{2}}.
\notag  \label{b1.16}
\end{eqnarray}
\end{proposition}

\begin{proof}
As in the proof of Proposition $2$, the involute $\overline{x}=\overline{x}%
(s)$ of the curve $x$ in $\mathbb{E}^{4}$ has the parametrization
\begin{equation*}
\overline{x}(s)=x(s)+(c-s)V_{1}(s),
\end{equation*}%
where $V_{1}$ is the unit tangent vector of $x$.

Further, the differentiation of the position vector $\overline{x}(s)$
implies that
\begin{eqnarray}
\overline{x}^{\prime }(s) &=&\varphi V_{2},  \notag \\
\overline{x}^{\prime \prime }(s) &=&-\varphi \kappa _{1}V_{1}+\varphi
^{\prime }V_{2}+\varphi \kappa _{2}V_{3},  \label{b1.17} \\
\overline{x}^{\prime \prime \prime }(s) &=&-\left \{ \left( \kappa
_{1}\varphi \right) ^{\prime }+\kappa _{1}\varphi ^{\prime }\right \}
V_{1}+\left \{ \varphi ^{\prime \prime }-\kappa _{1}^{2}\varphi -\kappa
_{2}^{2}\varphi \right \} V_{2}  \notag \\
&&+\left \{ \left( \kappa _{2}\varphi \right) ^{\prime }+\kappa _{2}\varphi
^{\prime }\right \} V_{3}+\varphi \kappa _{2}\kappa _{3}V_{4},  \notag
\end{eqnarray}%
where $\varphi =(c-s)\kappa _{1}$ is a differentiable function.
Consequently, substituting
\begin{eqnarray}
A &=&\kappa _{1}^{\prime }\varphi +2\kappa _{1}\varphi ^{\prime }  \notag \\
B &=&\varphi ^{\prime \prime }-\kappa _{1}^{2}\varphi -\kappa _{2}^{2}\varphi
\label{b1.18} \\
C &=&\kappa _{2}^{\prime }\varphi +2\kappa _{2}\varphi ^{\prime }  \notag \\
D &=&\varphi \kappa _{2}\kappa _{3},  \notag
\end{eqnarray}%
the last vector becomes
\begin{equation}
\overline{x}^{\prime \prime \prime }=-AV_{1}+BV_{2}+CV_{3}+DV_{4}.
\label{b1.19}
\end{equation}%
Furthermore, differentiating $\overline{x}^{\prime \prime \prime }$ with
respect to $s$, we get

\begin{equation}
\begin{array}{c}
\text{ \ }\overline{x}^{\prime \prime \prime \prime }=\text{\ }-\left\{
A^{\prime }+\kappa _{1}B\right\} V_{1}+\left\{ -\kappa _{1}A-\kappa
_{2}C+B^{\prime }\right\} V_{2} \\
+\left\{ \kappa _{2}B-\kappa _{3}D+C^{\prime }\right\} V_{3}+\left\{
D^{\prime }+\kappa _{3}C\right\} V_{4}.%
\end{array}
\label{b1.20}
\end{equation}%
Now, by the use of (\ref{b1.17}), we can compute the vector form $\overline{x%
}^{\prime }(s)\wedge \overline{x}^{\prime \prime }(s)\wedge \overline{x}%
^{\prime \prime \prime }(s)$ and second principal normal of \ $\overline{x}$
as in the following;%
\begin{equation*}
\overline{x}^{\prime }(s)\wedge \overline{x}^{\prime \prime }(s)\wedge
\overline{x}^{\prime \prime \prime }(s)=\varphi ^{2}\left\{ D\kappa
_{2}V_{1}+D\kappa _{1}V_{3}+\left( \kappa _{1}C-\kappa _{2}A\right)
V_{4}\right\}
\end{equation*}%
and
\begin{equation}
\overline{V}_{4}(s)=\frac{x^{\prime }(s)\wedge x^{\prime \prime }(s)\wedge
x^{\prime \prime \prime }(s)}{\left\Vert x^{\prime }(s)\wedge x^{\prime
\prime }(s)\wedge x^{\prime \prime \prime }(s)\right\Vert }=\frac{D\kappa
_{2}V_{1}+D\kappa _{1}V_{3}-\left( \kappa _{2}A-\kappa _{1}C\right) V_{4}}{W}
\label{b1.21}
\end{equation}%
where
\begin{equation}
W=\sqrt{D^{2}\left( \kappa _{1}^{2}+\kappa _{2}^{2}\right) +\left( \kappa
_{2}A-\kappa _{1}C\right) ^{2}}.  \label{b1.22}
\end{equation}%
Similarly, we can compute the vector form $\overline{V}_{4}(s)\wedge $ $%
\overline{x}^{\prime }(s)\wedge \overline{x}^{\prime \prime }(s)$ and first
principal normal $\overline{V}_{3}(s)$ of \ $\overline{x}$ as%
\begin{equation*}
\overline{V}_{4}(s)\wedge \overline{x}^{\prime }(s)\wedge \overline{x}%
^{\prime \prime }(s)=\frac{\varphi ^{2}}{W}\left\{ -\left( \kappa
_{2}A-\kappa _{1}C\right) \kappa _{2}V_{1}-\left( \kappa _{2}A-\kappa
_{1}C\right) \kappa _{1}V_{3}+D\left( \kappa _{1}^{2}+\kappa _{2}^{2}\right)
V_{4}\right\}
\end{equation*}%
and%
\begin{eqnarray}
\overline{V}_{3}(s) &=&\frac{\overline{V}_{4}(s)\wedge \overline{x}^{\prime
}(s)\wedge \overline{x}^{\prime \prime }(s)}{\left\Vert \overline{V}%
_{4}(s)\wedge \overline{x}^{\prime }(s)\wedge \overline{x}^{\prime \prime
}(s)\right\Vert }  \notag \\
&=&\frac{-\left( \kappa _{2}A-\kappa _{1}C\right) \kappa _{2}V_{1}-\left(
\kappa _{2}A-\kappa _{1}C\right) \kappa _{1}V_{3}+D\left( \kappa
_{1}^{2}+\kappa _{2}^{2}\right) V_{4}}{W\sqrt{\kappa _{1}^{2}+\kappa _{2}^{2}%
}}.  \label{b1.23}
\end{eqnarray}%
Finally, the vector form $\overline{V}_{3}(s)\wedge \overline{V}%
_{4}(s)\wedge $ $\overline{x}^{\prime }(s)$ and the normal $\overline{V}%
_{2}(s)$ of \ $\overline{x}$ becomes%
\begin{equation*}
\overline{V}_{3}(s)\wedge \overline{V}_{4}(s)\wedge \overline{x}^{\prime
}(s)=\varphi \left\{ D^{2}\left( \kappa _{1}^{2}+\kappa _{2}^{2}\right)
-\left( \kappa _{2}A-\kappa _{1}C\right) ^{2}\right\} \left( -\kappa
_{1}V_{1}+\kappa _{2}V_{3}\right)
\end{equation*}%
and%
\begin{equation}
\overline{V}_{2}(s)=\frac{\overline{V}_{3}(s)\wedge \overline{V}%
_{4}(s)\wedge \overline{x}^{\prime }(s)}{\left\Vert \overline{V}%
_{3}(s)\wedge \overline{V}_{4}(s)\wedge \overline{x}^{\prime }(s)\right\Vert
}=\frac{-\kappa _{1}V_{1}+\kappa _{2}V_{3}}{\sqrt{\kappa _{1}^{2}+\kappa
_{2}^{2}}}.  \label{b1.24}
\end{equation}%
Consequently, an easy calculation gives
\begin{eqnarray}
\left\langle \overline{V}_{2}(s),\overline{x}^{\prime \prime
}(s)\right\rangle &=&\varphi \sqrt{\kappa _{1}^{2}+\kappa _{2}^{2}}  \notag
\\
\left\langle \overline{V}_{3}(s),\overline{x}^{\prime \prime \prime
}(s)\right\rangle &=&\frac{W}{\sqrt{\kappa _{1}^{2}+\kappa _{2}^{2}}}
\label{b1.25} \\
\left\langle \overline{V}_{4}(s),\overline{x}^{\prime \prime \prime \prime
}(s)\right\rangle &=&-\frac{\left( \kappa _{2}A-\kappa _{1}C\right) \left(
\kappa _{3}C+D^{\prime }\right) +D\left( \kappa _{2}A^{\prime }-\kappa
_{1}C^{\prime }\right) +D^{2}\kappa _{1}\kappa _{3}}{W}.  \notag
\end{eqnarray}

Hence, from the relations (\ref{b1.25}) and (\ref{a1.5}) we deduce (\ref%
{b1.15}). This completes the proof of the proposition.
\end{proof}

For the case $x$ is a $W$-curve one can get the following results.

\begin{corollary}
\cite{TA} Let $\overline{x}$ be an involute of a generic $x$ curve in $%
\mathbb{E}^{4}$ given with the Frenet curvatures $\overline{\kappa }_{1},%
\overline{\kappa }_{2}$ and $\overline{\kappa }_{3}$. If $x$ is a $W$-curve
then the Frenet $4$-frame, $\overline{V}_{1},\overline{V}_{2},\overline{V}%
_{3}$ and $\overline{V}_{4}$ and the Frenet curvatures $\overline{\kappa }%
_{1}$, $\overline{\kappa }_{2}$ and $\overline{\kappa }_{3}$ of the involute
$\overline{x}$ of the curve $x$ are given by%
\begin{eqnarray}
\overline{V}_{1}(s) &=&V_{2},  \notag \\
\overline{V}_{2}(s) &=&\frac{-\kappa _{1}V_{1}+\kappa _{2}V_{3}}{\sqrt{%
\kappa _{1}^{2}+\kappa _{2}^{2}}}  \notag \\
\overline{V}_{3}(s) &=&V_{4}  \label{b1.26} \\
\overline{V}_{4}(s) &=&\frac{\kappa _{2}V_{1}+\kappa _{1}V_{3}}{\sqrt{\kappa
_{1}^{2}+\kappa _{2}^{2}}},  \notag
\end{eqnarray}
and
\begin{eqnarray}
\overline{\kappa }_{1} &=&\frac{\sqrt{\kappa _{1}^{2}+\kappa _{2}^{2}}}{%
\left \vert \varphi \right \vert }\text{,}  \notag \\
\overline{\kappa }_{2} &=&\frac{\kappa _{2}\kappa _{3}}{\left \vert \varphi
\right \vert \sqrt{\kappa _{1}^{2}+\kappa _{2}^{2}}},  \label{b1.27} \\
\overline{\kappa }_{3} &=&\frac{-\kappa _{1}\kappa _{3}}{\left \vert \varphi
\right \vert \sqrt{\kappa _{1}^{2}+\kappa _{2}^{2}}}  \notag
\end{eqnarray}%
respectively, where $\varphi =(c-s)\kappa _{1}$.
\end{corollary}

\begin{corollary}
Let $\overline{x}$ be an involute of a generic $x$ curve in $\mathbb{E}^{4}$
given with the Frenet curvatures $\overline{\kappa }_{1},\overline{\kappa }%
_{2}$ and $\overline{\kappa }_{3}.$ If $x$ is a $W$-curve then$\ \overline{x}
$ becomes a ccr-curve.
\end{corollary}

\subsection{Involutes of order $2$ in $\mathbb{E}^{4}$}

\begin{flushleft}
An involute of order $2$ of a regular curve $x$ in $\mathbb{E}^{4}$ has the
parametrization
\end{flushleft}

\begin{equation}
\overline{x}(s)=x(s)+\lambda _{1}(s)V_{1}(s)+\lambda _{2}(s)V_{2}(s)
\label{b1.29}
\end{equation}%
where $V_{1},V_{2}$ are tangent and normal vectors of $x$ in $\mathbb{E}^{4}$
and $\lambda _{1}$, $\lambda _{2}$ are differentiable functions satisfying
\begin{eqnarray}
\lambda _{1}^{\prime }(s)&=&\kappa _{1}(s)\lambda _{2}(s)-1,  \notag \\
\lambda _{2}^{\prime }(s)&=&-\lambda _{1}(s)\kappa _{1}(s).  \label{b1.30}
\end{eqnarray}

As in the previous subsection we get the following result.

\begin{corollary}
Let $x=x(s)$ be a unit speed Salkowski curve in $\mathbb{E}^{4}$. Then the
involute $\overline{x}$ of order $2$ of the curve $x$ has the
parametrization (\ref{b1.29}) given with the coefficient functions
\begin{eqnarray}
\lambda _{1}(s) &=&c_{1}\sin (\kappa _{1}s)+c_{2}\cos (\kappa _{1}s),  \notag
\\
\lambda _{2}(s) &=&c_{1}\cos (\kappa _{1}s)-c_{2}\sin (\kappa _{1}s)-\frac{1%
}{\kappa _{1}}.  \label{b1.31}
\end{eqnarray}%
where $c_{1}$ and $c_{2}$ are real constants.
\end{corollary}

We obtain the following result.

\begin{proposition}
Let $x=x(s)$ be a regular curve in $\mathbb{E}^{4}$ given with nonzero
Frenet curvatures $\kappa _{1}$, $\kappa _{2}$ and $\kappa _{3}.$ Then
Frenet $4$-frame, $\overline{V}_{1},\overline{V}_{2},\overline{V}_{3}$ and $%
\overline{V}_{4}$ and Frenet curvatures $\overline{\kappa }_{1}$, $\overline{%
\kappa }_{2}$ and $\overline{\kappa }_{3}$ of the involute $\overline{x}$ of
order $2$ of a regular curve $x$ in $\mathbb{E}^{4}$ are given by%
\begin{eqnarray}
\overline{V}_{1}(s) &=&V_{3},  \notag \\
\overline{V}_{2}(s) &=&\frac{-\kappa _{2}V_{2}+\kappa _{3}V_{4}}{\sqrt{%
\kappa _{2}^{2}+\kappa _{3}^{2}}},  \notag \\
\overline{V}_{3}(s) &=&\frac{K\left( \kappa _{2}^{2}+\kappa _{3}^{2}\right)
V_{1}+\left( \kappa _{2}N-\kappa _{3}L\right) \kappa _{3}V_{2}+\left( \kappa
_{2}N-\kappa _{3}L\right) \kappa _{2}V_{4}}{W\sqrt{\kappa _{2}^{2}+\kappa
_{3}^{2}}},  \label{b1.32} \\
\overline{V}_{4}(s) &=&\frac{\left( \kappa _{2}N-\kappa _{3}L\right)
V_{1}+\kappa _{3}KV_{2}+\kappa _{2}KV_{4}}{W},  \notag
\end{eqnarray}%
and
\begin{eqnarray}
\overline{\kappa }_{1} &=&\frac{\sqrt{\kappa _{2}^{2}+\kappa _{3}^{2}}}{%
\left\vert \phi \right\vert };\text{ }\phi :=\lambda _{2}(s)\kappa _{2}(s)
\notag \\
\text{ }\overline{\kappa }_{2} &=&\frac{W}{\phi ^{2}\left( \kappa
_{2}^{2}+\kappa _{3}^{2}\right) },  \label{b1.33} \\
\text{ }\overline{\kappa }_{3} &=&\frac{\left( \kappa _{2}N-\kappa
_{3}L\right) \left( \kappa _{1}L+K^{\prime }\right) +\left( \kappa
_{2}N^{\prime }-\kappa _{3}L^{\prime }\right) K+\kappa _{1}\kappa _{3}K^{2}}{%
W\phi ^{4}\overline{\kappa }_{1}\overline{\kappa }_{2}}  \notag
\end{eqnarray}%
where%
\begin{eqnarray*}
K &=&\kappa _{1}\kappa _{2}\phi \\
L &=&2\kappa _{2}\phi ^{\prime }+\kappa _{2}^{\prime }\phi \\
N &=&2\kappa _{3}\phi ^{\prime }+\kappa _{3}^{\prime }\phi
\end{eqnarray*}%
and
\begin{eqnarray}
W &=&\sqrt{K^{2}\left( \kappa _{2}^{2}+\kappa _{3}^{2}\right) +\left( \kappa
_{2}N-\kappa _{3}L\right) ^{2}} \\
&=&\left\vert \phi \right\vert \sqrt{\kappa _{1}^{2}\kappa _{2}^{2}\left(
\kappa _{2}^{2}+\kappa _{3}^{2}\right) +\left( \kappa _{2}\kappa
_{3}^{\prime }-\kappa _{3}\kappa _{2}^{\prime }\right) ^{2}}.  \notag
\label{b1.34}
\end{eqnarray}
\end{proposition}

\begin{proof}
Let $\overline{x}=\overline{x}(s)$ be the involute of order $2$ of the curve
$x$ in $\mathbb{E}^{4}$. Then by the use of (\ref{b1.2}), we get
\begin{equation}
\overline{x}^{\prime }(s)=\phi V_{3}  \label{b1.35}
\end{equation}%
where $\phi =\lambda _{2}(s)\kappa _{2}(s)$ is a differentiable function.
Further, the differentiation of (\ref{b1.35}) implies that
\begin{eqnarray}
\overline{x}^{\prime \prime }(s) &=&-\phi \kappa _{2}V_{2}+\phi ^{\prime
}V_{3}+\phi \kappa _{3}V_{4},  \notag \\
\overline{x}^{\prime \prime \prime }(s) &=&\kappa _{1}\kappa _{2}\phi
V_{1}+\left \{ 2\kappa _{2}\phi ^{\prime }+\kappa _{2}^{\prime }\phi \right
\} V_{2},  \label{b1.36} \\
&&+\left \{ \phi ^{\prime \prime }-\kappa _{2}^{2}\phi -\kappa _{3}^{2}\phi
\right \} V_{3}+\left \{ 2\kappa _{3}\phi ^{\prime }+\kappa _{3}^{\prime
}\phi \right \} V_{4}.  \notag
\end{eqnarray}%
Consequently, substituting
\begin{eqnarray}
K &=&\kappa _{1}\kappa _{2}\phi  \notag \\
L &=&2\kappa _{2}\phi ^{\prime }+\kappa _{2}^{\prime }\phi  \label{b1.37} \\
M &=&\phi ^{\prime \prime }-\kappa _{2}^{2}\phi -\kappa _{3}^{2}\phi  \notag
\\
N &=&2\kappa _{3}\phi ^{\prime }+\kappa _{3}^{\prime }\phi  \notag
\end{eqnarray}%
the last vector becomes
\begin{equation}
\overline{x}^{\prime \prime \prime }=KV_{1}-LV_{2}+MV_{3}+NV_{4}.
\label{b1.38}
\end{equation}%
Furthermore, differentiating $\overline{x}^{\prime \prime \prime }$ with
respect to $s$ we get%
\begin{eqnarray}
\overline{x}^{\prime \prime \prime \prime } &=&\left \{ K^{\prime }+\kappa
_{1}L\right \} V_{1}+\left \{ \kappa _{1}K-\kappa _{2}M-L^{\prime }\right \}
V_{2}  \notag \\
&&+\left \{ M^{\prime }-\kappa _{2}L-\kappa _{3}N\right \} V_{3}+\left \{
N^{\prime }+\kappa _{3}M\right \} V_{4}  \label{b1.39}
\end{eqnarray}

Hence, substituting (\ref{b1.35})-(\ref{b1.39}) into (\ref{a1.4}) and (\ref%
{a1.5}), after some calculations as in the previous proposition, we get the
result.
\end{proof}

For the case $x$ is a $W$-curve then one can get the following results.

\begin{corollary}
Let $\overline{x}$ be an involute of order $2$ of a generic $x$ curve in $%
\mathbb{E}^{4}$ given with the Frenet curvatures $\overline{\kappa }_{1},%
\overline{\kappa }_{2}$ and $\overline{\kappa }_{3}$. If $x$ is a W-curve
then the Frenet $4$-frame, $\overline{V}_{1},\overline{V}_{2},\overline{V}%
_{3}$ and $\overline{V}_{4}$ and Frenet curvatures $\overline{\kappa }_{1}$,
$\overline{\kappa }_{2}$ and $\overline{\kappa }_{3}$ of the involute $%
\overline{x}$ of order $2$ of a regular curve $x$ in $\mathbb{E}^{4}$ are
given by%
\begin{eqnarray}
\overline{V}_{1}(s) &=&V_{3},  \notag \\
\overline{V}_{2}(s) &=&\frac{-\kappa _{2}V_{2}+\kappa _{3}V_{4}}{\sqrt{%
\kappa _{2}^{2}+\kappa _{3}^{2}}}  \notag \\
\overline{V}_{3}(s) &=&V_{1}  \label{b1.40} \\
\overline{V}_{4}(s) &=&\frac{\kappa _{3}V_{2}+\kappa _{2}V_{4}}{\sqrt{\kappa
_{2}^{2}+\kappa _{3}^{2}}},  \notag
\end{eqnarray}%
and
\begin{eqnarray}
\overline{\kappa }_{1} &=&\frac{\sqrt{\kappa _{2}^{2}+\kappa _{3}^{2}}}{%
\left\vert \phi \right\vert }\text{,}  \notag \\
\text{ }\overline{\kappa }_{2} &=&\frac{\kappa _{1}\kappa _{2}}{\left\vert
\phi \right\vert \sqrt{\kappa _{2}^{2}+\kappa _{3}^{2}}},  \label{b1.41} \\
\text{ }\overline{\kappa }_{3} &=&\frac{\kappa _{1}\kappa _{3}}{\left\vert
\phi \right\vert \sqrt{\kappa _{2}^{2}+\kappa _{3}^{2}}},  \notag
\end{eqnarray}%
holds, where $\phi (s)=\lambda _{2}(s)\kappa _{2}(s)$.
\end{corollary}

\begin{corollary}
Let $\overline{x}$ be an involute of order $2$ of a generic $x$ curve in $%
\mathbb{E}^{4}$ given with the Frenet curvatures $\overline{\kappa }_{1}$, $%
\overline{\kappa }_{2}$ and $\overline{\kappa }_{3}.$ If $x$ is a $W$-curve
then $\overline{x}$ becomes a ccr-curve.
\end{corollary}

\subsection{Involutes of order $3$ in $\mathbb{E}^{4}$}

\begin{flushleft}
An involute of order $3$ of a regular curve $x$ in $\mathbb{E}^{4}$ has the
parametrization
\end{flushleft}

\begin{equation}
\overline{x}(s)=x(s)+\lambda _{1}(s)V_{1}(s)+\lambda _{2}(s)V_{2}(s)+\lambda
_{3}(s)V_{3}(s)  \label{b1.42}
\end{equation}%
where%
\begin{equation}
\begin{array}{c}
\lambda _{1}^{\prime }(s)=\kappa _{1}(s)\lambda _{2}(s)-1, \\
\lambda _{2}^{\prime }(s)=\lambda _{3}\kappa _{2}-\lambda _{1}\kappa _{1} \\
\lambda _{3}^{\prime }(s)=-\lambda _{2}(s)\kappa _{2}(s).%
\end{array}
\label{b1.43}
\end{equation}

By solving the system of differential equations in (\ref{b1.43}) we get the
following result.

\begin{corollary}
Let $x=x(s)$ is a unit speed W-curve in $\mathbb{E}^{4}$. Then the involute $%
\overline{x}$ of order $3$ of the curve $x$ has the parametrization (\ref%
{b1.42}) given with the coefficient functions
\begin{eqnarray}
\lambda _{1}(s) &=&\frac{\kappa _{1}\left( c_{2}\sin (\kappa s)-c_{3}\cos
(\kappa s)\right) }{\sqrt{\kappa }}+\frac{c_{1}\kappa -\kappa _{2}^{2}s}{%
\kappa },  \notag \\
\lambda _{2}(s) &=&c_{2}\cos (\kappa s)-c_{3}\sin (\kappa s)+\frac{\kappa
_{1}}{\kappa },  \label{b1.44} \\
\lambda _{3}(s) &=&\frac{\kappa _{2}\left( c_{2}\sin (\kappa s)-c_{3}\cos
(\kappa s)\right) }{\sqrt{\kappa }}-\frac{c_{1}\kappa _{1}\kappa -\kappa
_{1}\kappa _{2}^{2}s}{\kappa \kappa _{2}},  \notag
\end{eqnarray}%
where $\kappa =\kappa _{1}^{2}+\kappa _{2}^{2},$ $c_{1}$, $c_{2}$ and $c_{3}$
are real constants$.$
\end{corollary}

We obtain the following result.

\begin{proposition}
Let $x=x(s)$ be a regular curve in $\mathbb{E}^{4}$ given with nonzero
Frenet curvatures $\kappa _{1},\kappa _{2}$ and $\kappa _{3}$. Then Frenet
Frenet $4$-frame, $\overline{V}_{1},\overline{V}_{2},\overline{V}_{3}$ and $%
\overline{V}_{4}$ and Frenet curvatures $\overline{\kappa }_{1}$, $\overline{%
\kappa }_{2}$ and $\overline{\kappa }_{3}$ of the involute $\overline{x}$ of
order $3$ of a regular curve $x$ in $\mathbb{E}^{4}$ are given by%
\begin{eqnarray}
\overline{V}_{1}(s) &=&V_{4},  \notag \\
\overline{V}_{2}(s) &=&-V_{3},  \label{b1.45} \\
\overline{V}_{3}(s) &=&V_{2},  \notag \\
\overline{V}_{4}(s) &=&V_{1},  \notag
\end{eqnarray}%
and
\begin{eqnarray}
\overline{\kappa }_{1} &=&\frac{\kappa _{3}}{\left\vert \psi \right\vert },
\notag \\
\overline{\kappa }_{2} &=&\frac{\kappa _{2}}{\left\vert \psi \right\vert },
\label{b1.46} \\
\overline{\kappa }_{3} &=&-\frac{\kappa _{1}}{\left\vert \psi \right\vert },
\notag
\end{eqnarray}%
where $\psi (s)=\lambda _{3}(s)\kappa _{3}(s)$.
\end{proposition}

\begin{proof}
Let $\overline{x}=\overline{x}(s)$ be the involute of order $3$ of the curve
$x$ in $\mathbb{E}^{4}$. Then by the use of (\ref{b1.2}) with (\ref{b1.3}),
we get
\begin{equation}
\overline{x}^{\prime }(s)=\psi V_{4}  \label{b1.47}
\end{equation}%
where $\psi =\lambda _{3}(s)\kappa _{3}(s)$ is a differentiable function.
Further, the differentiation of (\ref{b1.47}) implies that
\begin{eqnarray}
\overline{x}^{\prime \prime }(s) &=&-\psi \kappa _{3}V_{3}+\psi ^{\prime
}V_{4},  \notag \\
\overline{x}^{\prime \prime \prime }(s) &=&\kappa _{2}\kappa _{3}\psi
V_{2}-\left\{ 2\kappa _{3}^{\prime }\psi +\kappa _{3}^{\prime }\phi \right\}
V_{3}+\left\{ \psi ^{\prime \prime }-\kappa _{3}^{2}\psi \right\} V_{4}.
\notag
\end{eqnarray}%
Consequently, substituting
\begin{eqnarray}
E &=&\kappa _{2}\kappa _{3}\psi  \notag \\
F &=&2\kappa _{3}^{\prime }\psi +\kappa _{3}^{\prime }\phi  \label{b1.49} \\
G &=&\psi ^{\prime \prime }-\kappa _{3}^{2}\psi  \notag
\end{eqnarray}%
the last vector becomes
\begin{equation}
\overline{x}^{\prime \prime \prime }=EV_{2}-FV_{3}+GV_{4}.  \label{b1.50}
\end{equation}%
Furthermore, differentiating $\overline{x}^{\prime \prime \prime }$ with
respect to $s$ we get%
\begin{eqnarray}
\overline{x}^{\prime \prime \prime \prime } &=&-\kappa _{1}EV_{1}+\left\{
\kappa _{2}F+E^{\prime }\right\} V_{2}  \notag \\
&&+\left\{ \kappa _{2}E-\kappa _{3}G-F^{\prime }\right\} V_{3}+\left\{
G^{\prime }-\kappa _{3}F\right\} V_{4}.  \label{b1.51}
\end{eqnarray}

Hence, substituting (\ref{b1.47})-(\ref{b1.51}) into (\ref{a1.4}) and (\ref%
{a1.5}), after some calculations we get the result.
\end{proof}

\begin{corollary}
The involute $\overline{x}$ of order $3$ of a ccr-curve $x$ in $\mathbb{E}%
^{4}$ is also a ccr-curve of $\mathbb{E}^{4}$.
\end{corollary}

\section{Generalized Evolute Curves in $\mathbb{E}^{m+1}$}

Let $x=x(s)$ be a generic curve in $\mathbb{E}^{n}$ given with Frenet frame $%
V_{1},V_{2},V_{3}...,V_{n}$ and Frenet curvatures $\kappa _{1},\kappa
_{2},...,\kappa _{n-1}.$ For simplicity, we can take $n=m+1$, to construct
the Frenet frame $V_{1}=T,V_{2}=N_{1},V_{3}=N_{2}...,V_{n}=N_{m}$ and Frenet
curvatures $\kappa _{1},\kappa _{2},...,\kappa _{m}.$ The centre of the
osculating hypersphere of $x$ at a point lies in the hyperplane normal to
the $x$ at that point. The curve passing through the centers of the
osculating hyperspheres of $x$ defined by
\begin{equation}
\widetilde{x}=x+\overset{m}{\underset{i=1}{\sum }}c_{i}N_{i},  \label{c1.1}
\end{equation}%
which is called \textit{generalized evolute} (or focal curve) of $x$, where $%
c_{1},c_{2},\ldots ,c_{m}$ are smooth functions of the parameter of the
curve $x$. The function $c_{i}$ is called the $i^{th}$ focal curvature of $%
\gamma $. Moreover, the function $c_{1}$ never vanishes and $c_{1}=\frac{1}{%
k_{1}}$ \cite{V1}.

The differentiation of the equation (\ref{c1.1}) and the Frenet formulae (%
\ref{a1.1}) give the following equation\
\begin{eqnarray}
\widetilde{x}^{\prime }(s) &=&(1-\kappa _{1}c_{1})T+(c_{1}^{\prime } -\kappa
_{2}c_{2})N_{1}+  \notag \\
&+&\overset{m-1}{\underset{i=2}{\sum }}(c_{i-1}\kappa _{i}+c_{i}^{\prime }
-c_{i+1}\kappa _{i+1})N_{i}+(c_{m-1}\kappa _{m}+c_{m}^{\prime })N_{m}.
\label{c1.2}
\end{eqnarray}%
Since, the osculating planes of $\widetilde{x}$ are the normal planes of $x$%
, and the points of $\widetilde{x}$ are the center of the osculating sphere
of $x$ then the generalized evolutes $\widetilde{x}$ of the curve $x$ are
determined by
\begin{equation}
\left \langle \widetilde{x}^{\prime }(s),T(s)\right \rangle =\left \langle
\widetilde{x}^{\prime }(s),N_{1}(s)\right \rangle =...=\left \langle
\widetilde{x}^{\prime }(s),N_{m-1}(s)\right \rangle =0.  \label{c1.3}
\end{equation}%
This condition is satisfied if and only if
\begin{eqnarray}
1-\kappa _{1}c_{1} &=&0  \notag \\
c_{1}^{\prime }-\kappa _{2}c_{2} &=&0  \notag \\
&&\vdots  \label{c1.4} \\
c_{i-1}\kappa _{i}+c_{i}^{\prime }-c_{i+1}\kappa _{i+1} &=&0,\text{ \ }2\leq
i\leq m-1.  \notag
\end{eqnarray}%
hold. So, the focal curvatures of a curve parametrized by arclength $s$
satisfy the following "scalar Frenet equation" for c$_{m}\neq $ $0$ :%
\begin{equation}
\frac{R_{m}^{2}}{2c_{m}}=c_{m-1}\kappa _{m}+c_{m}^{\prime }  \label{c1.4*}
\end{equation}%
where
\begin{equation*}
R_{m}=\left \Vert \widetilde{x}-x\right \Vert =\sqrt{%
c_{1}^{2}+c_{2}^{2}+...+c_{m}^{2}}
\end{equation*}%
is the radius of the osculating $m$-sphere \cite{V1}. Consequently, the
generalized evolutes $\widetilde{x}$ of the curve $x$ are represented by the
formulas (\ref{c1.1}), and
\begin{equation}
\widetilde{x}^{\prime }(s)=(c_{m-1}\kappa _{m}+c_{m}^{\prime })N_{m}.
\label{c1.5}
\end{equation}%
If \ $\widetilde{x}^{\prime }(s)=0$, then $R_{m}$ is constant and the curve $%
x$ is spherical.

\begin{proposition}
\cite{V1}The curvatures of a generic curve $x=x(s):I\subset \mathbb{R}%
\rightarrow \mathbb{E}^{m+1}$ parametri- zed by arc length, may be obtained
in terms of the focal curvatures by the formula:%
\begin{equation}
\kappa _{i}=\frac{c_{1}c_{1}^{\prime }+c_{2}c_{2}^{\prime
}+...+c_{i-1}c_{i-1}^{\prime }}{c_{i-1}c_{i}}.  \label{c1.6}
\end{equation}
\end{proposition}

\begin{remark}
For a generic curve, the functions $c_{i}$ or $c_{i-1}$ can vanish at
isolated points. At these points the function $c_{1}c_{1}^{\prime
}+c_{2}c_{2}^{\prime }+...+c_{i-1}c_{i-1}^{\prime }$ also vanishes, and the
corresponding value of the function $\kappa _{i}$ may be obtained by
l'Hospital rule. Denote by $R_{m}$ the radius of the osculating $m$-sphere.
Obviously $R_{m}^{2}=c_{1}^{2}+c_{2}^{2}+...+c_{m}^{2}$ \cite{V1}.
\end{remark}

\begin{theorem}
\cite{V1} Let $x=x(s)$ be a generic curve in $\mathbb{E}^{m+1}$ given with
Frenet frame $T,N_{1},N_{2}...,N_{m}$ and Frenet curvatures $\kappa
_{1},\kappa _{2},...,\kappa _{m}.$ Then Frenet frame $\widetilde{T},%
\widetilde{N}_{1},\widetilde{N}_{2}...,\widetilde{N}_{m}$ and Frenet
curvatures $\widetilde{\kappa }_{1},\widetilde{\kappa }_{2},...,\widetilde{%
\kappa }_{m}$ of the generalized evolute $\widetilde{x}$ of $x$ in $\mathbb{E%
}^{m+1}$ are given by%
\begin{eqnarray}
\widetilde{T} &=&\epsilon N_{m}  \notag \\
\widetilde{N}_{k} &=&\delta _{k}N_{m-k};\text{ }1\leq k\leq m-1  \label{c1.7}
\\
\widetilde{N}_{m} &=&\pm T  \notag
\end{eqnarray}%
and%
\begin{equation}
\frac{\widetilde{\kappa }_{1}}{\left \vert \kappa _{m}\right \vert }=\frac{%
\widetilde{\kappa }_{2}}{\kappa _{m-1}}=...=\frac{\left \vert \widetilde{%
\kappa }_{m}\right \vert }{\kappa _{1}}=\frac{1}{\left \vert c_{m-1}\kappa
_{m}+c_{m}^{\prime }\right \vert }  \label{c1.8}
\end{equation}%
where $\epsilon (s)$ is the sign of $\left( c_{m-1}\kappa _{m}+c_{m}^{\prime
}\right) (s)$ and $\delta _{k}$ the sign of \ $(-1)^{k}\epsilon (s)\kappa
_{m}(s)$.
\end{theorem}

\subsection{Evolutes in $\mathbb{E}^{3}$}

\begin{flushleft}
An generalized evolute of a regular curve $x$ in $\mathbb{E}^{3}$ has the
parametrization
\end{flushleft}

\begin{equation}
\widetilde{x}(s)=x(s)+c_{1}(s)N_{1}(s)+c_{2}(s)N_{2}(s)  \label{c1.9}
\end{equation}%
where $N_{1}$ and $N_{2}$ are normal vectors of $x$ in $\mathbb{E}^{3}$ and $%
c_{1},c_{2}$ are focal curvatures satisfying
\begin{equation}
\text{ \ \ \ }c_{1}(s)=\frac{1}{\kappa _{1}(s)},\text{ }c_{2}(s)=\frac{\rho
^{\prime }(s)}{\kappa _{2}(s)}.  \label{c1.10}
\end{equation}%
where $\rho =c_{1}=\frac{1}{\kappa _{1}}$ is the radius of the curvature of $%
x$.

We obtain the following result.

\begin{proposition}
Let $x=x(s)$ be a regular curve in $\mathbb{E}^{3}$ given with nonzero
Frenet curvatures $\kappa _{1}$ and $\kappa _{2}$. Then Frenet curvatures $%
\widetilde{\kappa }_{1}$ and $\widetilde{\kappa }_{2}$ of the evolute $%
\widetilde{x}$ of the curve $x$ are given by
\begin{equation}
\widetilde{\kappa }_{1}=\frac{\kappa _{2}^{2}}{\left \vert \rho \kappa
_{2}^{2}+\rho ^{\prime }\right \vert },\text{ }\widetilde{\kappa }_{2}=\frac{%
\kappa _{1}\kappa _{2}}{\left \vert \rho \kappa _{2}^{2}+\rho ^{\prime
}\right \vert }.  \label{c1.11}
\end{equation}%
where $\rho =\frac{1}{\kappa _{1}}$ is the radius of the curvature of $x$.
\end{proposition}

\begin{proof}
As a consequence of (\ref{c1.8}) we get (\ref{c1.9}).
\end{proof}

\begin{corollary}
The evolute $\widetilde{x}$ of a generalized helix in $\mathbb{E}^{3}$ is
also a generalized helix in $\mathbb{E}^{3}$.
\end{corollary}

By the use of (\ref{c1.4*}) with (\ref{c1.10}) one can get the following
result.

\begin{corollary}
A regular curve with nonzero curvatures $\kappa _{1}$ and $\kappa _{2}$ lies
in a sphere if and only if
\begin{equation}
\left( \frac{\rho ^{\prime }}{\kappa _{2}}\right) ^{\prime }+\rho \kappa
_{2}=0  \label{c1.12}
\end{equation}%
holds, where $\rho =\frac{1}{\kappa _{1}}$ is the radius of the curvature of
$x$.
\end{corollary}

\subsection{Evolutes in $\mathbb{E}^{4}$}

\begin{flushleft}
An generalized evolute of a generic curve $x$ in $\mathbb{E}^{4}$ has the
parametrization
\end{flushleft}

\begin{equation}
\widetilde{x}(s)=x(s)+c_{1}(s)N_{1}(s)+c_{2}(s)N_{2}(s)+c_{3}(s)N_{3}(s)
\label{c1.13}
\end{equation}%
where $N_{1}$, $N_{2}$ and $N_{3}$ are normal vectors of $x$ in $\mathbb{E}%
^{4}$ and $c_{1}$, $c_{2}$ and $c_{3}$ are focal curvatures satisfying
\begin{equation}
\text{ \ \ \ }c_{1}(s)=\frac{1}{\kappa _{1}(s)},\text{ }c_{2}(s)=\frac{\rho
^{\prime }(s)}{\kappa _{2}(s)},c_{3}(s)=\frac{\rho (s)\kappa _{2}(s)+\left(
\frac{\rho ^{\prime }(s)}{\kappa _{2}(s)}\right) ^{\prime }}{\kappa _{3}(s)}.
\label{c1.14}
\end{equation}%
where $\rho =\frac{1}{\kappa _{1}}$ is the radius of the curvature of $x.$

We obtain the following result.

\begin{proposition}
Let $x=x(s)$ be a regular curve in $\mathbb{E}^{4}$ given with nonzero
Frenet curvatures $\kappa _{1},\kappa _{2}$ and $\kappa _{3}.$ Then Frenet $%
4 $-frame, $\widetilde{T},\widetilde{N}_{1},\widetilde{N}_{2}$ and $%
\widetilde{N}_{3}$ and Frenet curvatures $\widetilde{\kappa }_{1}$, $%
\widetilde{\kappa }_{2}$ and $\widetilde{\kappa }_{3}$ of the evolute $%
\widetilde{x}$ of a regular curve $x$ in $\mathbb{E}^{4}$ are given by%
\begin{eqnarray}
\widetilde{T}(s) &=&N_{3},  \notag \\
\widetilde{N}_{1}(s) &=&-N_{2},  \label{c1.15} \\
\widetilde{N}_{2}(s) &=&N_{1},  \notag \\
\widetilde{N}_{3}(s) &=&T,  \notag
\end{eqnarray}%
and
\begin{eqnarray}
\widetilde{\kappa }_{1} &=&\frac{\kappa _{3}}{\left \vert \psi \right \vert }%
,  \notag \\
\widetilde{\kappa }_{2} &=&\frac{\kappa _{2}}{\left \vert \psi \right \vert }%
,  \label{c1.16} \\
\widetilde{\kappa }_{3} &=&-\frac{\kappa _{1}}{\left \vert \psi \right \vert
}  \notag
\end{eqnarray}%
where $\psi (s)=c_{2}(s)\kappa _{3}(s)+c_{3}^{\prime }(s)$ is a smooth
function.
\end{proposition}

\begin{proof}
As a consequence of \ (\ref{c1.7}) with (\ref{c1.8}) we get the result.
\end{proof}

\begin{corollary}
The evolute $\widetilde{x}$ of a ccr-curve $x$ in $\mathbb{E}^{4}$ is also a
ccr-curve of $\mathbb{E}^{4}$.
\end{corollary}

By the use of (\ref{c1.5}) with (\ref{c1.10}) one can get the following
result.

\begin{corollary}
A regular curve with nonzero curvatures $\kappa _{1},\kappa _{2}$ and $%
\kappa _{3}$ lies on a sphere if and only if
\begin{equation}
\left( \frac{\rho (s)\kappa _{2}(s)+\left( \frac{\rho ^{\prime }(s)}{\kappa
_{2}(s)}\right) ^{\prime }}{\kappa _{3}(s)}\right) ^{\prime }+\rho ^{\prime
}(s)\frac{\kappa _{3}(s)}{\kappa _{2}(s)}=0  \label{c1.17}
\end{equation}%
holds, where $\rho =\frac{1}{\kappa _{1}}$ is the radius of the curvature.
\end{corollary}

\begin{proposition}
\cite{Mo} A curve $x=x(s):I\subset \mathbb{R}\rightarrow \mathbb{E}^{4}$ is
spherical, i.e., it is contained in a sphere of radius $R$, if and only if $%
x $ can be decomposed as%
\begin{equation}
x(s)=m-\frac{R}{\kappa _{1}}N_{1}(s)+\frac{R\kappa _{1}^{\prime }}{\kappa
_{2}\kappa _{1}^{2}}N_{2}(s)+\frac{R}{\kappa _{3}}\left( \frac{\kappa
_{1}^{\prime }}{\kappa _{2}\kappa _{1}^{2}}\right) ^{\prime }N_{3}(s).
\label{c1.18}
\end{equation}%
where $m$ is the center of the sphere.
\end{proposition}

\end{document}